\documentclass[11pt]{amsart}

\usepackage{amsmath,amssymb,amsthm,amsfonts,amscd,array,graphicx}
\usepackage{mathrsfs}
\usepackage[T1]{fontenc}
\usepackage[utf8]{inputenc}
\usepackage{fullpage}
\usepackage[unicode]{hyperref}
\hypersetup{pdfpagemode = UseNone}
\hypersetup{pdfstartview = FitH}
\hypersetup{pdfborder={0 0 0.5}}

\newcommand{\rk}{\operatorname{rk}}
\newcommand{\Gal}{\operatorname{Gal}}
\newcommand{\C}{\mathbb{C}}

\newcommand{\K}{\mathbb{K}}

\newcommand{\A}{\mathbb{A}}
\newcommand{\V}{\mathbb{V}}
\newcommand{\Jph}{\left|J(\varphi)\right|}
\newcommand{\Kx}{\K[\overline{x}]} 
\newcommand{\Kf}{\K[\overline{f}]}
\newcommand{\Kox}{\K(\overline{x})} 
\newcommand{\Kof}{\K(\overline{f})}

\newtheorem{theorem}{Theorem}[subsection]
\newtheorem{propos}[theorem]{Proposition}
\newtheorem{lem}[theorem]{Lemma}
\newtheorem{cor}[theorem]{Corollary}

\theoremstyle{definition}
\newtheorem{defin}{Definition}[section]

\newtheorem{ex}[theorem]{Example}

\title{Contracted divisors and Degree-Two Maps}
\author{Anton Trushin}
\email{Trushin.ant.nic@yandex.ru}
\thanks{The work was supported by the Foundation for the Advancement of Theoretical Physics and Mathematics “BASIS”}
\address{Faculty of Computer Science, HSE University, Pokrovsky Boulevard 11, Moscow, 109028 Russia}
\begin{document}
	\maketitle
	
\begin{abstract}
	We consider polynomial maps of affine space over an algebraically closed field of characteristic zero. We prove that every irreducible component of the zero locus of the Jacobian determinant corresponds to either a contracted divisor or a branching divisor. We further consider polynomial maps of degree two without contracted divisors and show that the Jacobian determinant is irreducible, anti-invariant under the Galois involution, and coincides with the defining equation of the unique branching divisor.
\end{abstract}

\section{Introduction}

Let $\K$ be an algebraically closed field of characteristic zero, and let $\A^n$ denote the affine space over $\K$.
Consider a polynomial map
\[
\varphi: \A^n \to \A^n, \qquad
\varphi = \overline{f} = (f_1,\ldots,f_n), \quad f_i \in \K[\overline{x}],
\]
where $\overline{x} = (x_1,\ldots,x_n)$.

Let
\[
J(\varphi)=\left(\frac{\partial f_i}{\partial x_j}\right)
\]
be the Jacobian matrix and write $\Jph=\det J(\varphi)$.

A polynomial map is called a \emph{Keller map} if $\Jph$ is a nonzero constant. 
The Jacobian Conjecture, posed in~\cite{Keller}, asserts that every Keller map is an automorphism. 
Despite considerable effort, the conjecture remains open even in dimension two.

One approach to the Jacobian Conjecture is to study the behaviour of irreducible polynomials under the dual homomorphism
\[
\varphi^*:\K(\overline{f}) \to \K(\overline{x}), \qquad \varphi^*(H)=H(\overline{f}).
\]
By~\cite[Theorem~3.7]{Bak}, a Keller map $\varphi:\C^n\to\C^n$ is
invertible if and only if $\varphi^*$ maps irreducible polynomials to
irreducible ones.
By~\cite[Theorem~5.1]{Jed}, this extends to arbitrary fields of
characteristic zero: a polynomial map is invertible if and only if
$\varphi^*$ preserves irreducibility.
Moreover, by~\cite[Theorem~5.2]{Jed}, for a Keller map the dual
homomorphism $\varphi^*$ sends square-free polynomials to square-free
polynomials.

In Section~\ref{zerosJac} we introduce two types of divisors naturally associated with polynomial maps of affine space: 
\emph{contracted divisors}, viewed as analogues of exceptional divisors of birational morphisms, 
and \emph{branching divisors}, analogous to branch loci of finite morphisms.
In Theorem~\ref{classif} we show that every irreducible component of the zero locus of~$\Jph$ corresponds to either a contracted divisor or a branching divisor.

This dichotomy is natural, but we are not aware of a reference where it is stated explicitly, so we include a proof.
As an application, we obtain an alternative proof of the result on square-free polynomials mentioned above.

We then consider polynomial maps of finite degree. 
Recall that the degree of $\varphi$ is defined as the degree of the field extension $\K(\overline{x})/\K(\overline{f})$.
By~\cite{BCW}, the Jacobian Conjecture holds whenever this extension
is Galois. 
In particular, this applies to maps of degree two. 
Further results for low degrees in dimension two were obtained in~\cite{O,DO,D}.

In Section~\ref{mofd2} we study polynomial maps of degree two.
In this case the field extension $\K(\overline{x})/\K(\overline{f})$ is Galois of degree two, and hence
\[
\Gal(\K(\overline{x}),\K(\overline{f}))=\{id,\tau\},
\]
where $\tau$ is a birational involution.

We relate the action of $\tau$ on $\K(\overline{x})$ to the geometry of the fiber product
\[
Z=\A^n\times_{\varphi}\A^n,
\]
which parametrizes pairs of points with the same image under $\varphi$. 
For maps of degree two, the variety $Z$ contains the graph of the involution $\tau$, while its remaining components correspond to subvarieties contracted by $\varphi$.

Using this description, we show that, in the absence of contracted divisors, the branching is governed by a single irreducible polynomial $s$ satisfying $\tau(s)=-s$. 
In particular, all irreducible divisors of the Jacobian coincide with this polynomial, and we prove in Theorem~\ref{mainth} that
\[
\Jph=s.
\]

Thus, in this setting, the Jacobian is anti-invariant under the nontrivial Galois involution.

The author is grateful to Alexander Perepechko and Dmitry Trushin for numerous useful discussions and remarks.

	\section{Preliminaries, notation and conventions}
	
	Let $\K$ be an algebraically closed field of characteristic zero, and let
	\[
	\varphi:\A^n \to \A^n, \qquad 
	\varphi = (\overline{f}) = (f_1,\ldots,f_n), \quad f_j \in \K[\overline{x}],
	\]
	where $\overline{x} = (x_1,\ldots,x_n)$, be a polynomial map.
	We denote by
	\[
	\varphi^*:\K[\overline{f}] \to \K[\overline{x}], \qquad 
	\varphi^*(H)=H(\overline{f}),
	\]
	the induced homomorphism of coordinate rings.
	
	More generally, for a homomorphism of function fields
	\[
	\alpha:\K(X)\to\K(Y)
	\]
	we denote by
	\[
	\alpha^*:X\dasharrow Y
	\]
	the corresponding rational map.
	
	Recall that the differential $d\varphi$ is the map of tangent spaces
	induced by the Jacobian matrix
	\[
	J(\varphi)=\left(\frac{\partial f_i}{\partial x_j}\right)
	=
	\left(\frac{\partial\overline{f}}{\partial\overline{x}}\right).
	\]
	For a variety $X$ we denote its tangent space by $T_X$, and by $T_{a,X}$ the tangent space at a point $a\in X$.
	We say that $d\varphi$ is degenerate at $a\in X$ if
	\[
	\dim\bigl(d_a\varphi(T_{a,X})\bigr)<\dim X.
	\]
	
	We use the notation
	\[
	\frac{\partial H(\overline{f})}{\partial\overline{x}}
	=
	\left(
	\frac{\partial H(f_1,\ldots,f_n)}{\partial x_1},
	\ldots,
	\frac{\partial H(f_1,\ldots,f_n)}{\partial x_n}
	\right).
	\]
	
	For a variety $Y$ we denote by $I(Y)$ its vanishing ideal, and for polynomials $h_1,\ldots,h_m$ we write
	\[
	\V(h_1,\ldots,h_m)
	\]
	for their common zero locus.
	
	By a \emph{general point} of a variety $X$ we mean a point belonging to a Zariski open dense subset of~$X$.
	
	We define the degree of $\varphi$ as
	\[
	\deg\varphi=[\K(\overline{x}):\K(\overline{f})].
	\]
	
 	Let $X$ be an irreducible variety and let
	\[
	\xi:X\dasharrow Y
	\]
	be a rational map.
	We define the graph $\Gamma_\xi$ as the closure in $X\times Y$ of
	$\{(x,\xi(x))\mid x\in U\}$, where~$U\subset X$ is any nonempty open subset on which $\xi$ is regular.
	
	\medskip
	
	We now recall several basic facts concerning subvarieties contracted by a polynomial map.
	
	An irreducible variety $X\subseteq\A^n$ is said to be \emph{contracted by $\varphi$}, or simply \emph{contracted} when $\varphi$ is clear from the context, if
	\[
	\dim X > \dim \varphi(X).
	\]
	A reducible variety is contracted if each of its irreducible components is contracted.
	
	If a variety $X$ is contracted, then the differential $d\varphi$ is degenerate at smooth points of $X$, and the Jacobian $\Jph$ vanishes on $X$. Indeed, if $Y=\overline{\varphi(X)}$ and $\dim Y<\dim X$, then for a general point $a\in X$ one has
	\[
	\dim\bigl(d_a\varphi(T_{a}X)\bigr)\le \dim Y<\dim X,
	\]
	which implies that the Jacobian matrix is singular along $X$. In particular, if $X=\V(h)$ is an irreducible hypersurface contracted by $\varphi$, then $\Jph$ is divisible by $h$.
	
	The converse implication follows from a standard result (see~\cite[Chapter~I, \S5, Lemma~2.4]{Shaf}), which implies that if $X$ is irreducible and the differential $d\varphi$ is degenerate on a dense subset of $X$, then $\dim \varphi(X)<\dim X$, and hence $X$ is contracted.
	
	\begin{propos}\label{cont}
		An irreducible variety $X$ is contracted by $\varphi$ if and only if the differential $d\varphi$ is degenerate at a general point of $X$.
	\end{propos}
Note that if $\varphi$ is birational, then contracted divisors coincide with exceptional divisors (see~\cite{Shaf}). However, in general, not every contraction is a contraction of a hypersurface.

\begin{ex}
	Consider the map $\varphi:\A^3\to\A^3$ given by
	\[
	\varphi=(yx^2+zx,y,z).
	\]
	If $(y,z)\neq(0,0)$, then a point has at most two preimages. Thus the only contracted subvariety is the line $L=\{(x,0,0)\}$.
\end{ex}
	
We also recall the following standard result, often referred to as the Jacobian criterion. The following conditions are equivalent:
\begin{enumerate}
	\item the map $\varphi$ is dominant, equivalently, $\A^n$ is not contracted by $\varphi$;
	
	\item the Jacobian determinant $\Jph$ is not the zero polynomial;
	
	\item the polynomials $f_1,\ldots,f_n$ are algebraically independent.
\end{enumerate}
	
	Moreover, if $\varphi$ is dominant, then $\varphi^*$ is injective and induces a finite field extension $\K(\overline{f})\subset \K(\overline{x}).$
	
	We will also use the following standard fact (see~\cite[Proposition~7.16]{Harris}): the number of preimages of a general point under a rational map $\xi$ is finite if and only if the induced homomorphism $\xi^*$ gives a finite extension of function fields $\K(X)/\K(Y)$; in this case, the number of preimages of a general point is equal to the degree of this extension.
	
\section{Zeros of the Jacobian}\label{zerosJac}

\subsection{Contractions and the dual map}

We now study contracted divisors from the point of view of the dual map $\varphi^*$.

\begin{propos}\label{frac}
	Let $\varphi:\A^n\to\A^n$ be a polynomial map without contracted divisors. Then the following statements hold.
	
	\begin{enumerate}
		\item If $H_1,H_2\in\Kf$ are relatively prime, then 
		$H_1(\overline{f})$ and $H_2(\overline{f})$ are relatively prime.
		
		\item $\Kx\cap\Kof=\Kf$.
		
		\item If $h,g\in\Kx$ are relatively prime and $\frac{h}{g}\in\Kof$, then $h,g\in\Kf$.
	\end{enumerate}
\end{propos}

\begin{proof}
	(1) Suppose that 
	$
	g=\gcd(H_1(\overline{f}),H_2(\overline{f}))
	$
	is nonconstant. Then
	\[
	\varphi(\V(g))
	\subseteq
	\V(H_1)\cap\V(H_2)
	=
	\V(H_1,H_2).
	\]
	Since $H_1$ and $H_2$ are relatively prime, we have
	$
	\dim\V(H_1,H_2)<n-1.
	$
	Hence
	\[
	\dim\varphi(\V(g))<n-1=\dim\V(g),
	\]
	so $\V(g)$ is contracted, a contradiction.
	
	(2) Let
	$
	h=\frac{H_1(\overline{f})}{H_2(\overline{f})}\in\Kx\cap\Kof,
	$
	where $\frac{H_1}{H_2}$ is reduced. By~(1), the polynomials
	$
	H_1(\overline{f})
	$
	and
	$
	H_2(\overline{f})
	$
	are relatively prime. Since $h$ is a polynomial, the denominator must be constant, hence $h\in\Kf$.
	
	(3) Write
	$
	\frac{h}{g}
	=
	\frac{H(\overline{f})}{G(\overline{f})},
	$
	where $H,G\in\Kf$ and $\frac{H}{G}$ is reduced. By~(1), the polynomials
	$
	H(\overline{f})
	$
	and~$G(\overline{f})$ are relatively prime. Therefore the numerators and denominators of $\frac{h}{g}$ and $\frac{H(\overline{f})}{G(\overline{f})}$ differ only by a nonzero constant, hence $h,g\in\Kf$.
\end{proof}

\begin{propos}\label{gcd1}
	Let an irreducible hypersurface $X=\V(h)$ be contracted by $\varphi$. 
	Then there exist relatively prime irreducible polynomials $H_1,H_2\in\Kf$ such that both $\varphi^*(H_1)$ and $\varphi^*(H_2)$ are divisible by $h$.
\end{propos}

\begin{proof}
	Let $Y=\overline{\varphi(X)}$. 
	Then $Y$ is irreducible. 
	Let $I(Y)$ be generated by polynomials $G_1,\ldots,G_m$. 
	Since $Y$ is irreducible, we may assume that each $G_j$ is irreducible. 
	By assumption $\dim Y<n-1$, hence among these generators there exist two relatively prime ones.
\end{proof}

To illustrate the phenomenon of contracted divisors, we give a simple example.

\begin{ex}\label{exxy}
	Let
	\[
	\varphi:\A^2\to\A^2,\qquad 
	\varphi=(xy,y)=(f_1,f_2).
	\]
	Then $f_1$ and $f_2$ are relatively prime in $\Kf$, but in $\Kx$ they have a common divisor:
	\[
	\gcd(xy,y)=y.
	\]
	In accordance with Proposition~\ref{frac}, the line
	\[
	L=\{(x,0)\}
	\]
	is contracted by $\varphi$, namely $\varphi(L)=(0,0)$.
	
	Note also that the fraction $\frac{f_1}{f_2}$ is reduced in $\K(f_1,f_2)$ but not reduced in $\K(x,y)$ (cf.~\ref{frac}), and
	\[
	\K[x,y]\cap\K(f_1,f_2)\neq\K[f_1,f_2],
	\]
	since
	\[
	\frac{f_1}{f_2}=x\in\K[x,y]\cap\K(f_1,f_2)
	\]
	(cf.~\ref{frac}).
	
	As an illustration of Proposition~\ref{frac}, consider $h=x$, $g=1$. 
	They are relatively prime, and
	\[
	\frac{x}{1}=\frac{f_1}{f_2}\in\Kof,
	\]
	but $h=x\notin\Kf$.
\end{ex}

\subsection{Branching}

We now describe the second type of special hypersurfaces, namely branching divisors. 
At the end of the section we show that the zero locus of the Jacobian consists of contractions and branchings.

\begin{defin}
	Let $h \in \Kx$ be irreducible. Let $H \in \Kf$ be an irreducible polynomial such that~$H(\overline{f})$ is divisible by $h$. 
	Then $H$ is called an \emph{image polynomial} of $h$ and is denoted by $\varphi_*(h)$.
\end{defin}

\begin{propos}
	The image polynomial $\varphi_*(h)$ exists for every irreducible $h$. 
	Moreover, if~$X = \V(h)$ is not contracted, then $\varphi_*(h)$ is uniquely determined up to multiplication by a nonzero constant.
\end{propos}

\begin{proof}
	The variety $\varphi(X)$ is irreducible as the image of an irreducible variety. 
	Hence $\varphi_*(h)$ can be chosen as any irreducible polynomial in the ideal of $\overline{\varphi(X)}$.
	
	Suppose now that $X$ is not contracted. Then $\dim \varphi(X) = n-1$, and therefore $Y=\overline{\varphi(X)}$ is an irreducible hypersurface. 
	Hence $Y = \V(H)$ for some irreducible $H \in \Kf$, and this polynomial $H$ is unique up to a nonzero scalar multiple.
\end{proof}

\begin{defin}\label{defcreas}
	Let $h$ be irreducible and let $X = \V(h)$ be not contracted by $\varphi$. 
	Then $X$ is called a \emph{branching divisor} for $\varphi$ if $\varphi^*(\varphi_*(h))$ is divisible by $h^2$.
\end{defin}

By Corollary~\ref{cont}, a variety $X$ is contracted by $\varphi$ if and only if the differential $d\varphi$ is degenerate on the tangent space at a general point of $X$. 
We now show that a non-contracted hypersurface $X=\V(h)$ is a branching divisor if and only if the differential is degenerate on the conormal direction to $\varphi(X)$.

\begin{propos}\label{branch}
	Let $h\in\Kx$ be irreducible, let $X=\V(h)$, and let $H=\varphi_*(h)$. 
	The following conditions are equivalent:
	\begin{enumerate}
		\item $X$ is a branching divisor;
		
		\item $X$ is not contracted and
		$
		\frac{\partial H(\overline{f})}{\partial\overline{f}}
		\cdot
		J(\varphi)=0
		$
		on $X$.
	\end{enumerate}
	
	Moreover, these conditions imply that $\Jph=0$ on $X$.
\end{propos}

\begin{proof}
	Suppose first that $X$ is a branching divisor. Then
	$
	H(\overline{f})=h^2h_0
	$
	for some $h_0\in\Kx$. By the chain rule,
	\[
	\frac{\partial H(\overline{f})}{\partial\overline{x}}
	=
	\frac{\partial H(\overline{f})}{\partial\overline{f}}
	\cdot
	J(\varphi).
	\]
	On the other hand,
	\[
	\frac{\partial H(\overline{f})}{\partial\overline{x}}
	=
	2hh_0\frac{\partial h}{\partial\overline{x}}
	+
	h^2\frac{\partial h_0}{\partial\overline{x}},
	\]
	hence
	\[
	\frac{\partial H(\overline{f})}{\partial\overline{f}}
	\cdot
	J(\varphi)
	=
	h\left(
	2h_0\frac{\partial h}{\partial\overline{x}}
	+
	h\frac{\partial h_0}{\partial\overline{x}}
	\right),
	\]
	so the product vanishes on $X$.
	
	Since $H$ is irreducible, the vector
	$
	\frac{\partial H}{\partial\overline{f}}
	$
	vanishes only on the singular locus of $\V(H)$, which has dimension at most $n-2$. Since $X$ is not contracted, the image $\varphi(X)$ is not contained in this singular locus. Hence
	$
	\frac{\partial H(\overline{f})}{\partial\overline{f}}\neq0
	$
	on a dense open subset of $X$, and therefore
	$
	\frac{\partial H(\overline{f})}{\partial\overline{f}}
	\cdot
	J(\varphi)=0
	$
	implies that $J(\varphi)$ is singular on a dense open subset of $X$. Thus
	$
	\Jph=0
	$
	on $X$.
	
	Conversely, suppose that $X$ is not contracted and
	$
	\frac{\partial H(\overline{f})}{\partial\overline{f}}
	\cdot
	J(\varphi)=0
	$
	on $X$. Write
	$
	H(\overline{f})=hh_0.
	$
	Differentiating, we obtain
	\[
	\frac{\partial H(\overline{f})}{\partial\overline{x}}
	=
	\frac{\partial H(\overline{f})}{\partial\overline{f}}
	\cdot
	J(\varphi)
	=
	h_0\frac{\partial h}{\partial\overline{x}}
	+
	h\frac{\partial h_0}{\partial\overline{x}}.
	\]
	Restricting to $X$ gives
	$
	h_0\frac{\partial h}{\partial\overline{x}}=0.
	$
	Since $h$ is irreducible,
	$
	\frac{\partial h}{\partial\overline{x}}
	$
	does not vanish identically on $X$. Hence $h$ divides $h_0$, and therefore $H(\overline{f})$ is divisible by $h^2$. Thus $X$ is a branching divisor.
\end{proof}

\subsection{Classification theorem}

We now relate the behavior of the differential and its dual on the zero locus of the Jacobian.

\begin{lem}\label{classiflem}
	Let $h$ be an irreducible divisor of $\Jph$, let $X = \V(h)$, and let~$H = \varphi_*(h)$. 
	Then for a point $a \in X_{reg}$ at least one of the following holds:
	\begin{enumerate}
		\item $\displaystyle \frac{\partial H(\overline{f})}{\partial\overline{f}}\cdot J(\varphi) = 0$,
		\item $\displaystyle \dim\!\left(d_a\varphi(T_{a,X})\right) \leq n - 2$.
	\end{enumerate}
\end{lem}

\begin{proof}
	Let $T(a)$ be a matrix whose columns form a basis of the tangent space $T_{a,X}$ at a point~$a\in X_{reg}$. Then
	\[
	\rk T(a)=n-1,
	\qquad
	\rk\!\left(J(\varphi)\cdot T(a)\right)
	=\dim\!\left(d_{a}\varphi(T_{a,X})\right),
	\]
	and
	\[
	\frac{\partial h}{\partial \overline{x}}\cdot T(a)=0.
	\]
	
	Let $H(\overline{f})=h h_0$. Differentiating gives
	\[
	\frac{\partial H(\overline{f})}{\partial\overline{x}}
	=
	\frac{\partial H(\overline{f})}{\partial\overline{f}}\cdot J(\varphi)
	=
	h_0\frac{\partial h}{\partial\overline{x}}
	+
	h\frac{\partial h_0}{\partial\overline{x}} .
	\]
	
	Restricting to $X$ we obtain
	\[
	\frac{\partial H(\overline{f})}{\partial\overline{f}}\cdot J(\varphi)
	=
	h_0\frac{\partial h}{\partial\overline{x}} .
	\]
	
	Multiplying on the right by $T(a)$ and using 
	$\frac{\partial h}{\partial\overline{x}}\cdot T(a)=0$, we get
	\[
	\frac{\partial H(\overline{f})}{\partial\overline{f}}
	\cdot J(\varphi)\cdot T(a)
	=0.
	\]
	
	Applying Frobenius' inequality to the product
	\[
	\frac{\partial H(\overline{f})}{\partial\overline{f}}\cdot J(\varphi)\cdot T(a),
	\]
	we obtain
	\[
		\rk\!\left(
		\frac{\partial H(\overline{f})}{\partial\overline{f}}
		\cdot J(\varphi)
		\right)
		+
		\rk\!\left(J(\varphi)\cdot T(a)\right)	\le		\rk\!\left(J(\varphi)\right).
	\]
	
	Since $\Jph=0$ on $X$, we have $\rk J(\varphi)\le n-1$, and therefore
	\[
	\rk\!\left(
	\frac{\partial H(\overline{f})}{\partial\overline{f}}
	\cdot J(\varphi)
	\right)
	+
	\rk\!\left(J(\varphi)\cdot T(a)\right)
	\le n-1.
	\]
	
	Hence, if
	\[
	\frac{\partial H(\overline{f})}{\partial\overline{f}}\cdot J(\varphi)\neq 0,
	\]
	then
	\[
	\rk\!\left(J(\varphi)\cdot T(a)\right)\le n-2.
	\]
\end{proof}

\begin{theorem}\label{classif}
	Every irreducible component of the zero locus of $\Jph$ is either contracted
	by $\varphi$ or is a branching divisor.
\end{theorem}

\begin{proof}
	Let $h$ be an irreducible divisor of $\Jph$, let
	$X=\V(h)$, and let $H=\varphi_*(h)$. Consider $X_{reg}$.
	
	By Lemma~\ref{classiflem}, for every point $a\in X_{reg}$ one of the two conditions holds.
	
	Since $X$ is irreducible, one of these conditions holds for a general point on $X$.
	
	If
	\[
	\frac{\partial H(\overline{f})}{\partial\overline{f}}\cdot J(\varphi)=0,
	\]
	then by Proposition~\ref{branch}, $X$ is a branching divisor.
	
	If instead
	\[
	\dim\!\left(d_{a}\varphi(T_{a,X})\right)\le n-2
	\]
	for a general point $a\in X$, then by Corollary~\ref{cont}, $X$ is contracted.
\end{proof}

Note that if $\varphi$ is finite, then by Theorem \ref{classif} points on a branching divisor are precisely branch points in the usual sense (see~\cite{Shaf}).
We now apply these results to Keller maps.

\begin{propos}\label{Kmap}
	Let $\varphi=(\overline f)$ and suppose $\Jph\in\K^\times$. Then
	
	\begin{enumerate}
		\item if $H_1,H_2\in\Kf$ are coprime, then $\varphi_*(H_1)$ and $\varphi_*(H_2)$ are also coprime;
		
		\item $\Kx\cap\Kof=\Kf$;
		
		\item if an irreducible polynomial $H\in\Kf$ is square-free, then $\varphi^*(H)$ is also square-free.
	\end{enumerate}
\end{propos}

\begin{proof}
	Since $\Jph$ never vanishes, the map $\varphi$ has no contracted divisors. 
	Thus (1) and (2) follow from Proposition~\ref{frac}. 
	Statement (3) follows from Proposition~\ref{branch}.
\end{proof}

As a corollary we obtain Keller's theorem.

\begin{cor}[Keller's theorem]
	Let $\varphi=(\overline f)$, $\Jph\in\K^\times$, and suppose $\Kof=\Kox$. Then~$\varphi$ is invertible.
\end{cor}

\begin{proof}
	By Proposition~\ref{Kmap}(2) we have
	\[
	\Kf=\Kx\cap\Kof=\Kx\cap\Kox=\Kx,
	\]
	hence $\varphi^*$ is an isomorphism.
\end{proof}

We now give an alternative proof of~\cite[Corollary~2.2]{BY}.

\begin{cor}
	The map $\varphi$ is a Keller map if and only if the dual map $\varphi^*$ sends square-free polynomials to square-free polynomials.
\end{cor}

\begin{proof}
	Suppose $\Jph\in\K^\times$ and let $H\in\Kf$ be square-free. Write
	\[
	H=H_1\cdots H_m
	\]
	as a product of irreducible polynomials.
	
	By Proposition~\ref{Kmap}(3), each $\varphi^*(H_j)$ is square-free, and by (1) the polynomials $\varphi^*(H_j)$ and~$\varphi^*(H_k)$ are coprime for $j\ne k$.
	
	Now suppose $\Jph\notin\K$. Let $h$ be an irreducible divisor of $\Jph$. By Theorem~\ref{classif}, $h$ defines either a contraction or a branching divisor.
	
	If $h$ defines a contraction, then by Proposition~\ref{gcd1} there exist irreducible coprime polynomials~$H_1,H_2\in\Kf$ such that both $\varphi^*(H_1)$ and $\varphi^*(H_2)$ are divisible by $h$. Hence $\varphi^*(H_1H_2)$ is divisible by $h^2$.
	
	If $h$ defines branching, then $\varphi^*(\varphi_*(h))$ is divisible by $h^2$ by Definition~\ref{defcreas}.
\end{proof}

\section{Degree-Two Maps}\label{mofd2}

In this section we show that the Jacobian of a degree-two map without contracted divisors is an irreducible semi-invariant of the Galois involution.

\subsection{Galois Involution and The Fiber Product}

For a dominant polynomial map $\varphi:\A^n\to\A^n$, the field extension
\[
\K(\overline{f})\subset\K(\overline{x})
\]
is finite, and its degree coincides with the number of points in a general fiber of $\varphi$.

Here we consider maps of degree two. In this case the extension $\Kox/\Kof$ is Galois, and the Galois group
\[
\Gal(\Kox,\Kof)=\{id,\tau\}
\]
consists of two elements, where $\tau$ is a birational involution preserving $\Kof$. For the dual birational map we fix the notation $\theta=\tau^*$.

We record the following observation.

\begin{lem}\label{2fib}
	If $\varphi$ is a degree-two map, then for a general point $a_1$ we have
	\[
	\varphi^{-1}(\varphi(a_1))=\{a_1,a_2\},\qquad 
	\theta(a_1) = a_2.
	\]
\end{lem}

\begin{proof}
	A general point has exactly two preimages, hence
	\[
	\varphi^{-1}(\varphi(a_1))=\{a_1,a_2\}.
	\]
	
	Since $\tau$ preserves $\Kof$, we have $\tau(\overline{f})=\overline{f}$. Passing to the dual map $\theta$, we obtain $\varphi\circ\theta=\varphi$. Substituting $a_1$ gives
	\[
	\varphi(\theta(a_1))=\varphi(a_1).
	\]
	Since $\tau \ne id$, it follows that for a general point $\theta(a_1)$ is the second point in the fiber, hence
	\[
	\theta(a_1) = a_2.
	\]
\end{proof}

Note that $\theta$ being defined at a point $a$ does not imply, in
general, that
\[
\varphi^{-1}(\varphi(a))=\{a,\theta(a)\}.
\]

\begin{ex}
	Let $\varphi=(xy,y^2)$. This map is the composition of the birational map $(xy,y)$ and the map $(x,y^2)$, hence $\deg\varphi=2$. 
	
	The map $\theta=(-x,-y)$ preserves $(xy,y^2)$ and is therefore the Galois involution. In this case $\theta$ is defined on all points of $\A^2$, while
	\[
	\varphi^{-1}(\varphi(1,0))=\{(x,0)\}\neq\{(1,0),(-1,0)\}.
	\]
\end{ex}

Now let
\[
Z=\A^n\times_\varphi \A^n \subset \A^{2n}
\]
be the fiber product. It consists of pairs of points with the same image:
\[
\varphi(\overline{x}_1)=\varphi(\overline{x}_2).
\]

For a degree-two map we may consider the graph $\varGamma=\varGamma_\theta$ of the map $\theta$, consisting of points of the form $(a_1,\theta(a_1))$. By Lemma~\ref{2fib} we have
\[
\varphi(a_1)=\varphi(\theta(a_1))=\varphi(a_2),
\]
and therefore $\varGamma\subset Z$.

The variety $Z$ also contains the diagonal $\Delta$, consisting of points of the form~$(\overline{x},\overline{x})$. Note that $Z$ is defined by $n$ equations in the $2n$-dimensional space $\A^{2n}$, hence every irreducible component of $Z$ has dimension at least $n$.

We now describe the tangent space $T_Z$. Consider the defining equations~$\varphi(\overline{x}_1)=\varphi(\overline{x}_2)$. Writing them as
\[
f_j(\overline{x}_1)-f_j(\overline{x}_2)=0
\]
and taking partial derivatives we obtain
\begin{align*}
	\frac{\partial(f_j(\overline{x}_1)-f_j(\overline{x}_2))}{\partial x_{1,k}}
	&=\frac{\partial f_j(\overline{x}_1)}{\partial x_{1,k}},\\
	\frac{\partial(f_j(\overline{x}_1)-f_j(\overline{x}_2))}{\partial x_{2,k}}
	&=-\frac{\partial f_j(\overline{x}_2)}{\partial x_{2,k}}.
\end{align*}

Arranging these partial derivatives into rows, we obtain the normal matrix $N$ of size $n\times 2n$:
\[
N=
\begin{pmatrix}
	J(\varphi)(\overline{x}_1) & -J(\varphi)(\overline{x}_2)
\end{pmatrix}.
\]

For points $(a_1,a_2)$ such that $\Jph(a_j)\neq0$, the matrix $N$ has full rank, and hence the dimension of the tangent space at such points is at most $n$.

Since both $\varGamma$ and $\Delta$ have dimension $n$, they are irreducible components of $Z$. Let $Z'$ denote the union of the remaining components. Then
\[
Z=\Delta\cup\varGamma\cup Z'.
\]

Consider the projections $\pi_1,\pi_2$ given by
\[
\pi_j(a_1,a_2)=a_j.
\]

\begin{propos}\label{zdecomp}
	Let $X=\pi_j(Z')\subset\A^n$. Then
	\begin{enumerate}
		\item $\dim X\le n-1$;
		\item $X$ is contracted by the map $\varphi$.
	\end{enumerate}
\end{propos}

\begin{proof}
	(1) For a general point $a \in \A^n$, its fiber $\pi_1^{-1}(a)$ consists of points of the form~$(a,b)$ such that $b \in \varphi^{-1}(\varphi(a))$. Moreover, if $a=b$ then $(a,b)\in\Delta$, while if~$a\neq b$ then $(a,b)\in\varGamma$. Hence for a general point the fiber $\pi_j^{-1}(a)$ contains no points of~$Z'$, and therefore the dimension of $X=\pi_j(Z')$ is at most $n-1$.
	
	(2) The dimension of $Z'$ is at least $n$, while $\dim X \le n-1$. Hence for a point~$a\in X$ the fiber $\pi_1^{-1}(a)$ is at least one-dimensional. Let
	\[
	Y=\pi_2(\pi_1^{-1}(a)).
	\]
	Then $Y$ is also at least one-dimensional and $\varphi(Y)=\varphi(a)$. Since $Y$ is contracted to a point, the differential $d\varphi$ is degenerate on the tangent space $T_Y$ at general points. Since $T_Y \subset T_X$, the differential is also degenerate on the tangent space to~$X$. Hence~$X$ is contracted by $\varphi$ by Corollary~\ref{cont}.
\end{proof}

\begin{cor}\label{dimcont}
	If $\varphi$ has no contracted divisors, then $\dim\pi_j(Z') < n-1$.
\end{cor}

\begin{ex}
	Consider the map
	\[
	\varphi:\A^2\to\A^2,\qquad \varphi(x,y)=(xy,y^2).
	\]
	This map is the composition of the birational isomorphism $(xy,y)$ and the degree-two map $(x,y^2)$, hence $\deg\varphi=2$.
	
	The variety $Z=\A^2\times_\varphi\A^2$ is defined by the equations
	\begin{align*}
		x_1y_1 &= x_2y_2,\\
		y_1^2 &= y_2^2 .
	\end{align*}
	
	Thus $Z$ is the union of three two-dimensional components:
	\[
	Z=\Delta\cup\varGamma\cup Z',
	\]
	where
	\[
	\Delta=\V(x_1-x_2,y_1-y_2),\qquad
	\varGamma=\V(x_1+x_2,y_1+y_2),\qquad
	Z'=\V(y_1,y_2).
	\]
	
	Under the projection $\pi_1$, the component $Z'$ maps to the line
	\[
	L=\{(x,0)\},
	\]
	which is contracted to the point $(0,0)$ by $\varphi$.
\end{ex}
To illustrate Corollary~\ref{dimcont} and the fact that the absence of contracted divisors does not imply the absence of the component $Z'$, we give another example.
\begin{ex}\label{4ex}
	Consider the map $\varphi:\A^4\to\A^4$ defined by
	\[
	\varphi
	\begin{pmatrix}
		x\\y\\z\\w
	\end{pmatrix}
	=
	\begin{pmatrix}
		wx^2+zy\\
		zx+wy\\
		z\\
		w
	\end{pmatrix}.
	\]
	
	We first show that this map has degree two. Consider the preimage of a point $(x_0,y_0,z_0,w_0)$ with $z_0\neq0$ and $w_0\neq0$:
	\[
	\begin{cases}
		wx^2+zy=x_0\\
		zx+wy=y_0\\
		z=z_0\\
		w=w_0
	\end{cases}
	\Longleftrightarrow
	\begin{cases}
		wx^2-\frac{z^2}{w}x=x_0-\frac{z}{w}y_0\\
		zx+wy=y_0\\
		z=z_0\\
		w=w_0
	\end{cases}
	\]
	
	The equation
	\[
	wx^2-\frac{z^2}{w}x-x_0+\frac{z}{w}y_0=0
	\]
	is quadratic in $x$ and has two distinct roots provided its discriminant
	\[
	D=\frac{z^4}{w^2}-4(zy_0-wx_0)
	\]
	is nonzero. Hence a general point has exactly two preimages. In particular, if $(z_0,w_0)\neq(0,0)$ then the point has at most two preimages, and hence $\varphi$ has no contracted divisors.
	
	At the same time the two-dimensional plane $\V(z,w)$ is contracted to the point $(0,0)$.
	
	The variety $Z$ contains the two-dimensional component
	\[
	Z'=\V(z_1,w_1,z_2,w_2).
	\]
	Its image under $\pi_1$ is the contracted plane $\V(z,w)$, which has dimension two.
\end{ex}

\subsection{Semi-invariants}

We begin by constructing an element $s \in \Kox$ such that $\tau(s) = -s$. 
For this, we take any $h \in \Kox \setminus \Kof$ and set
\[
s = h - \tau(h).
\]

\begin{lem}\label{semiinv1}
	Let $\varphi$ be a degree-two map. Then there exists $s \in \Kx$ such that $\tau(s) = -s$ and~$S(\overline{f}) = s^2 \in \Kf$ is square-free.
\end{lem}

\begin{proof}
	Take a nonzero $\widetilde{s} \in \Kox$ such that $\tau(\widetilde{s}) = -\widetilde{s}$ and write $\widetilde{S}(\overline{f}) = \widetilde{s}^2$. 
	
	Factor $\widetilde{S}$ as
	\[
	\widetilde{S} = S_1 \cdots S_m S_0^2,
	\]
	where $S_1,\ldots,S_m$ are pairwise coprime irreducible polynomials and $S_0 \in \Kof$. 
	Dividing $\widetilde{s}$ by $S_0$, we obtain an element $s = \widetilde{s}/S_0$ such that $s^2$ is square-free.
\end{proof}

\begin{cor}\label{creas}
	For a degree-two map we have $\Jph \notin \K^\times$.
\end{cor}

\begin{proof}
	Assume that $\varphi$ has no contracted divisors. Let $s_0$ be an irreducible divisor of $S_j(\overline{f})$. Then~$\varphi_*(s_0) = S_j$. 
	
	Since $S_j(\overline{f})$ is a square in $\Kx$, it is divisible by $s_0^2$, so $s_0$ defines a branching divisor. By Proposition~\ref{branch}, $s_0$ divides $\Jph$.
\end{proof}

We now study how branching is reflected in the action of $\tau$ on $\Kox$. 
From now on we assume that $\varphi$ has no contracted divisors.

\begin{lem}\label{deg}
	Let $h \in \Kx$. The following conditions are equivalent:
	\begin{enumerate}
		\item $\tau(h) = -h$;
		\item $h \notin \Kf$ and $h^2 \in \Kf$;
		\item $h \notin \Kof$ and there exists $m$ such that $h^m \in \Kof$.
	\end{enumerate}
\end{lem}

\begin{proof}
	If $\tau(h) = -h$, then $\tau(h)h = -h^2$ is invariant under $\tau$, hence $h^2 \in \Kof$, and therefore~$h^2 \in \Kf$ by Proposition~\ref{frac}.
	
	The implication $(2) \Rightarrow (3)$ is immediate with $m=2$.
	
	Finally, suppose $h^m \in \Kof$. Writing $\tau(h) = \widehat{h}$, we get $h^m = \widehat{h}^m$, so $h = \lambda \widehat{h}$ for some $m$-th root of unity $\lambda$. Since $\tau$ has order two, $\lambda = \pm 1$. Since $h \notin \Kof$,
	we conclude $\lambda = -1$.
\end{proof}

\begin{lem}\label{semiinv}
	Let $h \in \Kx \setminus \Kf$ be irreducible and assume that $h^2 \in \Kf$. Then
	\[
	\varphi^*(\varphi_*(h)) = h^2
	\qquad\text{and}\qquad
	\tau(h) = -h.
	\]
\end{lem}

\begin{proof}
	Let $H = \varphi^*(\varphi_*(h))$. Then $H$ divides $h^2$. Since $h \notin \Kf$, we must have $H = h^2$. The relation~$\tau(h) = -h$ follows from Lemma~\ref{deg}.
\end{proof}

We now view polynomials from $\Kx$ as functions on the variety $\varGamma$ via the projection $\pi_1$. Then $\theta$ acts on $\varGamma$ by
\[
\theta(a_1,a_2) = (a_2,a_1).
\]

\begin{lem}\label{invsemiinv}
	Let $g \in \Kx \cap \K[\varGamma]^\times$ be irreducible. Then either $g \in \Kf$ or $\tau(g) = -g$.
\end{lem}

\begin{proof}
	Let $G = \varphi_*(g)$ and write $G(\overline{f}) = g \cdot g_0$. 
	
	The images of $\V(g)$ and $\V(g_0)$ both contain open subsets of $\V(G)$, since there are no contracted divisors. Hence for a general point $a \in \V(g)$ there exists $a_0 \in \V(g_0)$ such that $\varphi(a) = \varphi(a_0)$, so~$(a,a_0) \in Z$.
	
	Since $g$ is invertible on $\varGamma$, such points cannot lie on $\varGamma$. Excluding the lower-dimensional part coming from $Z'$, we conclude that generically $(a,a_0)$ lies on the diagonal. Therefore $g$ and $g_0$ have the same irreducible divisors, and the claim follows from Lemma~\ref{deg}.
\end{proof}

\begin{cor}\label{invneirr}
	Let $g \in \K[\varGamma]^\times$. Then either $g \in \Kof$ or $\tau(g) = -g$.
\end{cor}

\begin{proof}
	Write $g$ as a quotient of two polynomials and apply Lemma~\ref{invsemiinv} to each irreducible factor.
\end{proof}

\begin{lem}\label{denom}
	Let $h \in \Kx$ and suppose
	\[
	\tau(h) = \frac{p}{r}
	\]
	is a reduced fraction. Then $r \in \Kf$.
\end{lem}

\begin{proof}
	Consider the sum
	\[
	h + \tau(h) = \frac{rh + p}{r}.
	\]
	Let $d$ be an irreducible divisor of $r$. Then $d$ does not divide $p$, and therefore does not divide the numerator $rh + p$. Since $h + \tau(h) \in \Kof$, it follows from Proposition~\ref{frac} that $r \in \Kf$.
\end{proof}

\begin{lem}\label{lemirr}
	Let $h_1 \in \Kx \setminus \K[\varGamma]^\times$ be irreducible. Then
	\[
	\tau(h_1) = gh_2,
	\]
	where $h_2$ is irreducible and noninvertible on $\varGamma$, and $g \in \K[\varGamma]^\times \cap \Kof$.
\end{lem}

\begin{proof}
	The function $h_1$ vanishes on $\varGamma$ but has no poles there, and the same holds for~$\tau(h_1)$. Therefore all noninvertible factors of $\tau(h_1)$ must appear in the numerator.
	
	Separating invertible and noninvertible factors, we obtain a decomposition of the required form. Irreducibility of $h_2$ follows from applying $\tau$ and comparing factors. That $g \in \Kf$ follows from Lemma~\ref{denom}.
\end{proof}

\begin{lem}\label{neib}
	Let $h_1 \in \Kx \setminus \K[\varGamma]^\times$ be irreducible and assume $h_1 \notin \Kf$. Then
	\[
	\varphi^*(\varphi_*(h_1)) = h_1 h_2,
	\]
	where $h_2$ is as in Lemma~\ref{lemirr}.
\end{lem}

\begin{proof}
	Using Lemma~\ref{lemirr}, we write $\tau(h_1) = h_2 \cdot g$ with $g$ invertible on $\varGamma$. Applying~$\tau$ again shows that $h_1 h_2$ is invariant, hence lies in $\Kf$.
	
	The images of $\V(h_1)$ and $\V(h_2)$ coincide, so their common image is defined by some irreducible~$H \in \Kf$. Comparing multiplicities shows that $\varphi^*(H) = h_1 h_2$.
\end{proof}

\begin{lem}\label{creasing2}
	Let $h$ be an irreducible polynomial defining a branching divisor. Then
	\[
	\varphi^*(\varphi_*(h)) = h^2
	\qquad\text{and}\qquad
	\tau(h) = -h.
	\]
\end{lem}

\begin{proof}
	Since $h$ defines branching, it is not contained in $\Kf$. If $h$ is invertible on $\varGamma$, the statement follows from previous results on invertible elements.
	
	Otherwise we apply Lemma~\ref{neib}. In this case branching forces $h_1 = h_2$, hence $\varphi^*(\varphi_*(h)) = h^2$. The relation $\tau(h) = -h$ follows from Lemma~\ref{deg}.
\end{proof}

\begin{propos}\label{semiinv2}
	Let $\varphi$ be a degree-two map without contracted divisors. Then there exists an irreducible $s \in \Kx$ such that $\tau(s) = -s$. Moreover, $s$ is unique up to multiplication by a constant.
\end{propos}

\begin{proof}
	By Corollary~\ref{creas}, a branching divisor exists. Let $s$ be an irreducible polynomial defining it. Then $\tau(s) = -s$ by Lemma~\ref{creasing2}.
	
	If $s_0$ is another such polynomial, then $\tau(s_0/s) = s_0/s$, hence $s_0/s \in \Kof$. By Proposition~\ref{frac},~$s_0/s \in \Kf$, so $s_0 = \lambda s$.
\end{proof}

\begin{cor}\label{jacdiv}
	If $\varphi$ has no contracted divisors, then the Jacobian $\Jph$ has a unique irreducible divisor $s$.
\end{cor}

\begin{proof}
	By Theorem~\ref{classif}, every irreducible divisor of the Jacobian defines branching. By Proposition~\ref{semiinv2}, such a divisor is unique.
\end{proof}

\subsection{Jacobian}

By Corollary~\ref{jacdiv}, if $\varphi$ has degree-two and no contracted divisors, then
\[
\Jph = s^k
\]
for some irreducible $s$ such that $\tau(s) = -s$ and some integer $k>0$. 
Fix such an $s$, and set~$S = s^2 \in \Kof$, which lies in $\Kf$ and is irreducible since $\varphi$ has no contracted divisors.

Consider the polynomial
\[
q=t^2-S \in \K[t,f_1,\ldots,f_n]
\]
and the hypersurface
\[
M=\V(q)\subset\A^{n+1}.
\]

Then $\varphi$ factors as $\varphi=\pi\circ\psi$, where
\[
\psi:\A^n\to M,\qquad
\psi(\overline{x})=(s(\overline{x}),\varphi(\overline{x})),
\]
and $\pi$ is the projection
\[
\pi(t,f_1,\ldots,f_n)=(f_1,\ldots,f_n).
\]

For $j=0,\ldots,n$, we define polynomial maps
\[
\varphi_0=\varphi,\qquad
\varphi_j=(f_1,\ldots,f_{j-1},s,f_{j+1},\ldots,f_n)\ \text{for } j\ge1.
\]

\begin{lem}\label{collinear}
	At smooth points of $M$ the vectors
	\[
	(s^k,|J(\varphi_1)|,\ldots,|J(\varphi_n)|)
	\]
	and
	\[
	\left(2s,
	\psi^*\!\left(\frac{\partial q}{\partial f_1}\right),
	\ldots,
	\psi^*\!\left(\frac{\partial q}{\partial f_n}\right)\right)
	\]
	are collinear.
\end{lem}

\begin{proof}
	The hypersurface $M$ is defined by the equation $q=t^2-S$, hence its normal vector at a smooth point is
	\[
	\left(2t,\frac{\partial q}{\partial f_1},\ldots,
	\frac{\partial q}{\partial f_n}\right).
	\]
	Pulling back via $\psi$, we obtain
	\[
	\left(2s,
	\psi^*\!\left(\frac{\partial q}{\partial f_1}\right),
	\ldots,
	\psi^*\!\left(\frac{\partial q}{\partial f_n}\right)\right).
	\]
	
	On the other hand, consider the Jacobian matrix
	\[
	J(\psi)=
	\begin{pmatrix}
		\frac{\partial s}{\partial x_1}&\cdots&\frac{\partial s}{\partial x_n}\\
		\frac{\partial f_1}{\partial x_1}&\cdots&\frac{\partial f_1}{\partial x_n}\\
		\vdots&\ddots&\vdots\\
		\frac{\partial f_n}{\partial x_1}&\cdots&\frac{\partial f_n}{\partial x_n}
	\end{pmatrix}.
	\]
	Its columns are tangent vectors to $M$. Let $(A_1,\ldots,A_{n+1})$ be the vector of signed maximal minors of $J(\psi)$. Then
	\[
	(A_1,\ldots,A_{n+1})\cdot J(\psi)=0,
	\]
	hence this vector is orthogonal to all tangent vectors and therefore defines a normal vector.
	
	Moreover, by construction of the minors we have
	\[
	A_1=|J(\varphi_0)|=\Jph,\qquad
	A_{j+1}=|J(\varphi_j)| \ \text{for } j=1,\ldots,n.
	\]
	
	Since the normal space to a hypersurface at a smooth point is one-dimensional, the two normal vectors are collinear.
\end{proof}

\begin{lem}\label{ngcd2}
	If $\V(s)$ is not contracted, then
	\[
	\gcd(|J(\varphi_0)|,|J(\varphi_1)|,\ldots,|J(\varphi_n)|)=1 .
	\]
\end{lem}

\begin{proof}
	For $j=1,\ldots,n$, denote by
	\[
	\pi_j:\A^n \to \A^{n-1}
	\]
	the projection forgetting the $j$-th coordinate.
	
	First note that for any $j>0$ the polynomial $s$ lies in~$\K[f_1,\ldots,f_{j-1},s,f_{j+1},\ldots,f_n]$, hence $s$ does not define branching for $\varphi_j$.
	Let $X=\overline{\varphi(\V(s))}$ and let $a\in X$ be a smooth point. The tangent space at $a$ is given by
	\[
	\alpha_1x_1+\ldots+\alpha_nx_n=0
	\]
	with $(\alpha_1,\ldots,\alpha_n)\neq0$. Choose $\alpha_j\neq0$. Then $\pi_j$ does not contract $X$ by Proposition~\ref{cont}, hence~$\varphi_j$ does not contract $X$.
	
	Thus $|J(\varphi_j)|$ is not divisible by $s$ by Theorem~\ref{classif}, while $|J(\varphi_0)|=\Jph$ is divisible only by $s$ by Corollary~\ref{jacdiv}.
\end{proof}

\medskip

\begin{theorem}\label{mainth}
	Let $\varphi:\A^n\to\A^n$, $\varphi=(\overline f)$ be a degree-two map without contracted divisors and suppose
	\[
	\Gal(\Kox,\Kof)=\{id,\tau\}.
	\]
	Then
	\[
	\Jph=s,
	\]
	where $s$ is irreducible and $\tau(s)=-s$.
\end{theorem}

\begin{proof}
	By Lemma~\ref{collinear}, there exist $\lambda_1,\lambda_2\in\Kox$ such that
	\[
	\lambda_1(s^k,|J(\varphi_1)|,\ldots,|J(\varphi_n)|)
	=
	\lambda_2\left(2s,
	\psi^*\!\left(\frac{\partial q}{\partial f_1}\right),
	\ldots,
	\psi^*\!\left(\frac{\partial q}{\partial f_n}\right)\right).
	\]
	
	Multiplying both sides by a common denominator, we may assume that~$\lambda_1,\lambda_2\in\Kx$. Dividing by their greatest common divisor, we may further assume that they are coprime.
	By Lemma~\ref{ngcd2}, $\gcd(|J(\varphi_0)|,\ldots,|J(\varphi_n)|)=1$,
	so $\lambda_2\in\K^\times$.
	Comparing the first coordinates, we obtain
	\[
	\lambda_1 s^k = 2\lambda_2 s.
	\]
	Hence $\lambda_1\in\K^\times$ and $k=1$.
	
\end{proof}

\begin{ex}
	As an application of Theorem~\ref{mainth}, we compute the Galois involution for the map from Example~\ref{4ex}:
	
	\[
	\varphi
	\begin{pmatrix}
		x\\y\\z\\w
	\end{pmatrix}
	=
	\begin{pmatrix}
		wx^2+zy\\
		zx+wy\\
		z\\
		w
	\end{pmatrix}.
	\]
	
	We have already shown that this map has degree two and has no contracted divisors. Hence Theorem~\ref{mainth} applies, and the Jacobian determinant is irreducible and anti-invariant under the Galois involution $\tau$.
	
	Since the last two coordinates are preserved by $\varphi$, we have
	\[
	\tau(z)=z, \qquad \tau(w)=w.
	\]
	
	We compute the Jacobian determinant:
	\[
	\Jph=2w^2x - z^2.
	\]
	By Theorem~\ref{mainth},
	\[
	\tau(2w^2x - z^2) = 2w^2\,\tau(x)-z^2=-(2w^2x - z^2),
	\]
	hence
	\[
	\tau(x)=-x+\frac{z^2}{w^2}.
	\]
	
	Since $f_2 = zx+wy$ is invariant under $\tau$, we obtain
	\[
	z\,\tau(x)+w\,\tau(y)=zx+wy.
	\]
	Substituting the expression for $\tau(x)$ gives
	\[
	-zx+\frac{z^3}{w^2}+w\,\tau(y)=zx+wy,
	\]
	and therefore
	\[
	\tau(y)=y+\frac{2z}{w}x-\frac{z^3}{w^3}.
	\]
	
	Thus the Galois involution is given by
	\[
	\tau(x,y,z,w)=\left(
	-x+\frac{z^2}{w^2},\;
	y+\frac{2z}{w}x-\frac{z^3}{w^3},\;
	z,\;
	w
	\right).\]
\end{ex}


\begin{thebibliography}{99}
	
	
	\bibitem{Bak}
	S.~Bakalarski,
	Jacobian problem for factorial varieties,
	\emph{Univ. Iagel. Acta Math.} \textbf{44} (2006), 31--34.
	
	\bibitem{BCW}
	H.~Bass, E.~H.~Connell, D.~Wright,
	The Jacobian conjecture: reduction of degree and formal expansion of the inverse,
	\emph{Bull. Amer. Math. Soc. (N.S.)} \textbf{7} (1982), no.~2, 287--330.
	
	\bibitem{BY}
	M.~M.~de~Bondt, D.~Yan,
	Irreducibility properties of Keller maps,
	\emph{Algebra Colloq.} \textbf{23} (2016), no.~4, 663--680.
	
	\bibitem{D}
	A.~V.~Domrina,
	Four-sheeted polynomial mappings in $\C^2$. The general case,
	\emph{Math. Notes} \textbf{65} (1999), no.~3, 386--389.
	
	\bibitem{DO}
	A.~V.~Domrina, S.~Y.~Orevkov,
	On four-sheeted polynomial mappings in $\C^2$. I. The case of an irreducible ramification curve,
	\emph{Math. Notes} \textbf{64} (1998), no.~6, 732--744.
	
	
	
	\bibitem{Harris}
	J.~Harris,
	\emph{Algebraic Geometry: A First Course},
	Graduate Texts in Mathematics, Vol.~133,
	Springer, New York, 1992.
	
	\bibitem{Jed}
	P.~J\k{e}drzejewicz,
	A characterization of Keller maps,
	\emph{J. Pure Appl. Algebra} \textbf{217} (2013), no.~1, 165--171.
	
	\bibitem{Keller}
	O.-H.~Keller,
	Ganze Cremona-Transformationen,
	\emph{Monatsh. Math. Phys.} \textbf{47} (1939), 299--306.
	
	\bibitem{O}
	S.~Y.~Orevkov,
	On three-sheeted polynomial mappings in $\C^2$,
	\emph{Math. USSR-Izv.} \textbf{29} (1987), no.~3, 587--596.
	
	\bibitem{Shaf}
	I.~R.~Shafarevich,
	\emph{Basic Algebraic Geometry I. Varieties in Projective Space},
	3rd ed., Springer, Berlin-Heidelberg, 2013.
	
	
\end{thebibliography}
\end{document}